\documentclass[reqno,10pt,a4paper,oneside]{amsart}

\usepackage{uniform-kan-prelude}
\usepackage{formatting-christian}

\title{Idempotent completion of cubes in posets}

\author{Christian Sattler}
\address{University of Gothenburg, Sweden}
\email{sattler@chalmers.se}

\date{December 23, 2018}

\newcommand{\Poset}{\mathbf{Poset}}
\newcommand{\FinComplPoset}{\Poset_{\mathsf{compl},\mathsf{fin}}}
\newcommand{\aug}{{\mathsf{aug}}}
\newcommand{\Kan}{{\mathsf{Kan}}}
\newcommand{\ttm}{{\mathsf{tt}}}
\newcommand{\fibrant}{{\mathsf{fib}}}

\begin{document}

\begin{abstract}
This note concerns the category $\Box$ of cartesian cubes with connections, equivalently the full subcategory of posets on objects $[1]^n$ with $n \geq 0$.
We show that the idempotent completion of $\Box$ consists of finite complete posets.
It follows that cubical sets, \ie presheaves over $\Box$, are equivalent to presheaves over finite complete posets.
This yields an alternative exposition of a result by Kapulkin and Voevodsky that simplicial sets form a subtopos of cubical sets.
\end{abstract}

\maketitle

\section{Preliminaries}

\subsection{Idempotents}

We briefly recall some of the basic theory of idempotents and their splittings.
A classical reference is~\cite{cauchy-completion}.

The following notions are relative to a category $\cal{E}$.
An \emph{idempotent} $(A, f)$ consists of an endomorphism $f \co A \to A$ satisfying $f^2 = f$:
\[
\xymatrix{
  A
  \ar[rr]^{f}
  \ar[dr]_{f}
&&
  A
\rlap{.}\\
&
  A
  \ar[ur]_{f}
}
\]
A \emph{retract} $(A, B, r, s)$ consists of a pair of morphisms $r \co A \to B$ (the retraction) and $s \co B \to A$ (the section) such that $rs = \id_B$:
\[
\xymatrix{
&
  A
  \ar[dr]^{r}
\\
  B
  \ar[ur]^{s}
  \ar[rr]_{\id}
&&
  B
\rlap{.}}
\]
Any retract $(A, B, s, r)$ induces an idempotent $sr \co A \to A$.
An idempotent is said to \emph{split} if it is induced by a retract diagram in this fashion.
The splitting of $f \co A \to A$ is also characterized as an equalizer or a coequalizer of $f$ and $\id_A$.
Retract diagrams are preserved by any functor, so these are absolute (co)limits.
The category $\cal{E}$ is called \emph{idempotent complete} if all its idempotents split.

\begin{definition}
A functor $F \co \cal{C} \to \cal{D}$ exhibits $\cal{D}$ as an \emph{idempotent completion} of $\cal{C}$ if $F$ is fully faithful, $\cal{D}$ is idempotent complete, and every object $B \in \cal{D}$ arises as a retract of an object $F A$ with $A \in \cal{C}$.
\qed
\end{definition}

The idempotent completion is also known as the Cauchy completion or Karoubi envelope.
It can be characterized using a universal property, but we will not need this here.

We write $\widehat{\cal{C}}$ for the category of presheaves over a category $\cal{C}$.
The main property of the idempotent completion relevant to our intended application is the following.

\begin{proposition}[{essentially \cite[Theorem~1]{cauchy-completion}}] \label{presheaves-idempotent-completion}
Let $F \co \cal{C} \to \cal{D}$ be an idempotent completion.
Then the induced pullback functor $F^* \co \widehat{\cal{D}} \to \widehat{\cal{C}}$ is an equivalence.
\qed
\end{proposition}

\subsection{Cube category}

We write $\Poset$ for the category of small posets.

Following \cite{kapulkin:cubical}, we define $\Box$ as the full subcategory of $\Poset$ on powers $[1]^n$ of the walking arrow $[1] \defeq \braces{0 \to 1}$ with $n \geq 0$.
Equivalently, in the style of \cite{cohen-et-al:cubicaltt}, we may view $\Box$ as the opposite of the category of bounded distributive lattices free on a finite set.
One may also explicitly describe $\Box$ as a category of cubes with symmetries, diagonals, and (upper and lower) connections.
Cubical sets using this notion of cube category are sufficiently rich to interpret all main features of cubical type theory as introduced in~\cite{cohen-et-al:cubicaltt}, including recent extensions to higher inductive types~\cite{coquand:hit}.

\section{Idempotent completion of $\Box$}

Let $\FinComplPoset$ denote the full subcategory of complete finite posets; note that it has a small (countable) skeleton.
Recall that every complete finite poset is also cocomplete.
Thus, we may also describe $\FinComplPoset$ as the full subcategory of posets on finite bounded lattices.

The walking arrow $[1]$ is a finite bounded lattice, hence are its finite powers $[1]^n$ with $n \geq 0$.
We thus obtain a fully faithful inclusion $\Box \to \FinComplPoset$.

\begin{theorem} \label{box-idempotent-completion}
The inclusion $\Box \to \FinComplPoset$ is an idempotent completion.
\end{theorem}

\begin{proof}
The category of finite posets is finitely complete, hence in particular idempotent complete.
To transfer this property to $\FinComplPoset$, it suffices to check given a retraction $A \to B$ of (finite) posets that $B$ is complete as soon as $A$ is.
This is \cref{poset-retract-complete} below.

It remains to check that every complete finite poset arises as a retract of $[1]^n$ with $n \geq 0$.
This is \cref{cocomplete-poset-retract-box} below.
\end{proof}

\begin{lemma} \label{poset-retract-terminal}
Retracts in $\Poset$ preserve the existence of terminal objects.
\end{lemma}

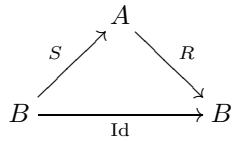
\begin{proof}
Let
\[
\xymatrix{
&
  A
  \ar[dr]^{R}
\\
  B
  \ar[ur]^{S}
  \ar[rr]_{\Id}
&&
  B
}
\]
be a retract of posets.
Given a terminal object $1_A$ in $A$, we will verify that $1_B \defeq R 1_A$ is terminal in $B$.
Given $X \in B$, a map $X \to 1_B$ is obtained by taking the map $S X \to 1_A$ (using terminality of $1_A$) and applying $R$.
\end{proof}

Note the \cref{poset-retract-terminal} really only works for posets, not categories in general.
For a counterexample, note that the terminal category is a retract of the walking pair of arrows.

\begin{lemma} \label{poset-retract-complete}
Retracts in $\Poset$ preserve completeness.
\end{lemma}

\begin{proof}
Let
\[
\xymatrix{
&
  A
  \ar[dr]^{R}
\\
  B
  \ar[ur]^{S}
  \ar[rr]_{\Id}
&&
  B
}
\]
be a retract of posets.
We assume $A$ complete and will show $B$ complete.

Given a diagram $F \co C \to B$, we need to show that $B \downarrow F$ has a terminal object.
By functoriality of comma category formation, the above retract diagram lifts to a retract diagram
\[
\xymatrix{
&
  A \downarrow SF
  \ar[dr]^{R \downarrow C}
\\
  B \downarrow F
  \ar[ur]^{S \downarrow C}
  \ar[rr]_{\Id}
&&
  B \downarrow F
}
\]
of comma categories, which are again posets.
By completeness of $A$, we have a terminal object in $A \downarrow SF$.
The conclusion then follows from \cref{poset-retract-terminal}.
\end{proof}

\begin{lemma} \label{cocomplete-poset-retract-box}
Let $C \in \Poset$ have set of objects $|C|$ and decidable hom-sets.
Then $C$ is a retract of $[1]^{|C|}$.
\end{lemma}

\begin{proof}
We use the posetal Yoneda embedding $y \co C \to [1]^{C^\op}$.
As for categories, it is the universal map from $C$ to a cocomplete poset.
Since $C$ is already cocomplete, there exists a unique cocontinuous functor $R$ making the following diagram commute (using that $C$ is skeletal):
\[
\xymatrix{
  C
  \ar[r]^-{y}
  \ar[dr]_{\Id}
&
  [1]^{C^\op}
  \ar@{.>}[d]^{R}
\\&
  C
\rlap{.}}
\]
This exhibits $C$ as a retract of $[1]^{C^\op}$.

In a second step, we show that $[1]^{C^\op}$ is a retract of $[1]^{|C|}$.
Seeing $|C|$ as a discrete poset, we have an inclusion $I \co |C| \to C$ that is bijective on objects.
Thus, the restriction functor $(-) \cc I \co [1]^{C^\op} \to [1]^{|C|}$ is fully faithful.
Together with its left adjoint $\Lan_I$, it thus forms a reflection.
Since $C$ is skeletal, the counit isomorphism of this reflection is valued in identities, \ie we have a retract
\[
\xymatrix{
&
  [1]^{|C|}
  \ar[dr]^{\Lan_I}
\\
  [1]^{C^\op}
  \ar[ur]^{(-) \cc I}
  \ar[rr]_{\Id}
&&
  [1]^{|C|}
\rlap{.}}
\]

The goal follows by composition of retracts.
\end{proof}

\section{Application to comparing cubical and simplicial sets}

\subsection{The essential geometric embedding}

Let $\Delta$ denote the simplex category.
We may use \cref{box-idempotent-completion} to give a different exposition of the proof by Kapulkin and Voevodsky of the following theorem:

\begin{theorem}[{\cite[Section~1]{kapulkin:cubical}}] \label{embedding-sset-to-cset}
There is an essential geometric embedding $\widehat{\Delta} \to \widehat{\Box}$.
\end{theorem}

Note that essentiality of the geometric embedding, meaning a further left adjoint to the inverse image functor, was not observed in the cited reference.

\begin{proof}
Observe first that every object $[n]$ of $\Delta$ is a complete finite poset.
We have fully faithful embeddings of index categories as follows:
\[
\xymatrix{
  \Delta
  \ar[dr]
&&
  \Box
  \ar[dl]^(0.4)*!/r0.5cm/{\labelstyle\text{idempotent completion}}
\\&
  \FinComplPoset
\rlap{,}}
\]
using \cref{box-idempotent-completion} for the annotation of the right arrow.

Upon taking presheaves, we therefore obtain essential geometric embeddings (with direct images given by right Kan extension, inverse images given by precomposition, and further left adjoints given by left Kan extension) as follows:
\begin{equation} \label{essential-geometric-embeddings}
\begin{gathered}
\xymatrix{
  \widehat{\Delta}
  \ar[dr]
&&
  \widehat{\Box}
  \ar[dl]^(0.4){\simeq}
\\&
  \widehat{\FinComplPoset}
\rlap{.}}
\end{gathered}
\end{equation}
Here, we used \cref{presheaves-idempotent-completion} to derive the equivalence on the right.
Thus, we obtain an essential geometric embedding $\widehat{\Delta} \to \widehat{\Box}$ as desired.
\end{proof}

Evaluating the resulting inverse image functor $\widehat{\Box} \to \widehat{\Delta}$, a representable $y([1]^n)$ is first mapped to the representable $y([1]^n)$ in $\widehat{\FinComplPoset}$ and then restricted to $\Delta$.
The resulting presheaf sends $[m] \in \Delta$ to $\Poset([m], [1]^n) \simeq \widehat{\Delta}(y[m], (\Delta^1)^n)$, \ie coincides with the standard triangulation of the $n$-cube.
By naturality and cocontinuity, the inverse image functor is thus the standard triangulation functor of cubical sets, guaranteeing that our construction coincides with the one of~\cite{kapulkin:cubical}.

The left adjoint $\widehat{\Delta} \to \widehat{\Box}$ to the inverse image functor sends a representable $y [m]$ first to the representable $y [m]$ in $\widehat{\FinComplPoset}$ and then restricts it to $\Box$.
The resulting presheaf sends $[1]^n$ to $\Poset([1]^n, [m])$.
Following the below remark, this may also be seen as the quotient of $y ([1]^m)$ by the endomorphism of $[1]^m$ reordering each $x$ in ascending fashion.

\begin{remark}
Write the poset $[n]$ as $\braces{0 \to \ldots \to n}$.
We can exhibit an object $[n] \in \Delta$ explicitly as a retract in $\Poset$ of $[1]^n$ by sending $k \in [n]$ to
\[
(\underbrace{0, \ldots, 0}_{n-k}, \underbrace{1, \ldots, 1}_{k}) \in [1]^n
\]
and $x \in [1]^n$ to $\sum_i x_i \in [n]$.
This observation suffices to recognize $\Delta$ as a full subcategory of the idempotent completion of $\Box$, using only the fact that $\Poset$ is idempotent complete and hence contains the idempotent completion of $\Box$ as a full subcategory.
In turn, this is enough to carry out the proof of \cref{embedding-sset-to-cset}, not requiring full use \cref{box-idempotent-completion} or its ingredients \cref{poset-retract-complete,poset-retract-terminal,cocomplete-poset-retract-box}.
We believe this mirrors best the original proof in~\cite{kapulkin:cubical}.
\qed
\end{remark}

\subsection{Preservation of monomorphisms}

Let us discuss some properties of the essential geometric embedding of \cref{embedding-sset-to-cset}.

\begin{proposition} \label{left-kan-extension-pres-mono-instance}
In the essential geometric embedding $\widehat{\Delta} \to \widehat{\Box}$, the left adjoint to the inverse image functor preserves monomorphisms.
\end{proposition}

Before we can show this statement, we will need to collect a few observations.
Recall from~\cite{bergner-rezk-elegant} the notion of elegant Reedy category.

\begin{lemma} \label{reedy-mono-lemma}
Let $\cal{C}$ be an elegant Reedy category.
Assume that $\cal{C}$ has pullbacks along face faces, and that these pullbacks preserve face and degeneracy maps.
Then the functor $\colim \co \widehat{\cal{C}} \to \Set$ preserves monomorphisms.
\end{lemma}

\begin{proof}
Consider a monomorphism
\[
\xymatrix@C-0.5cm{
  P
  \ar[dr]
  \ar@{->}[rr]^{F}
&&
  Q
  \ar[dl]
\\&
  \cal{C}
}
\]
of discrete Grothendieck fibrations over $\cal{C}$.
We have to show that $F$ induces a monomorphism on connected components.

Note that $F$ is a discrete fibration by standard closure properties; since it is mono, it is thus a sieve.
Note that $Q$ inherits the elegant Reedy structure from $\cal{C}$ and that the assumptions on pullbacks in $\cal{C}$ descend to $Q$.
Since $\cal{C}$ is elegant and $F$ is mono, $F$ has opcartesian lifts of degeneracy maps.

Let $S$ and $T$ be objects of $P$ such that $FS$ and $FT$ lie in the same connected component of $Q$, \ie are connected by a zig-zag of morphisms.
Using the Reedy factorization and pullbacks along face maps in $Q$, we may ``normalize'' this zig-zag to find face maps $\overline{S'} \to FS$ and $\overline{T'} \to FT$ such that $\overline{S'}$ and $\overline{T'}$ are connected by a zig-zag of degeneracy maps.%
\footnote{In fact, using elegancy, one may further reduce the zig-zag of degeneracy  maps to a cospan by taking pushouts.}
Since $F$ is a discrete Grothendieck fibration, we have lifts $S' \to S$ and $T' \to T$ of $\overline{S'} \to FS$ and $\overline{T'} \to FT$, respectively.
Since $F$ restricted to degeneracy maps in the base is a discrete Grothendieck bifibration as well as injective on objects, we may lift the zig-zag of degeneracy maps between $\overline{S'}$ and $\overline{T'}$ to a zig-zag between $S'$ and $T'$.
It thus follows that $S$ and $T$ lie in the same connected component.
\end{proof}

\begin{lemma} \label{left-kan-extension-pres-mono}
Let $\cal{A}$ be an elegant Reedy category and $i \co \cal{A} \to \cal{B}$ a functor.
Assume that $\cal{A}$ has pullbacks along face maps whenever the cospan under consideration lies in the image of the projection $M \downarrow i \to \cal{A}$ for some $M \in \cal{B}$, that these pullbacks preserve face and degeneracy maps, and that $i$ preserves these pullbacks.
Then left Kan extension $i_! \co \widehat{\cal{A}} \to \widehat{\cal{B}}$ preserves monomorphisms.
\end{lemma}

\begin{proof}
It will suffice to show that left Kan extension along $i$ followed by evaluation at $M \in \cal{B}$ preserves monomorphisms.
By the standard formula for left Kan extensions, this is given by the composite functor
\[
\xymatrix@C+0.2cm{
  \widehat{\cal{A}}
  \ar[r]^-{p^*}
&
  \widehat{M \downarrow i}
  \ar[r]^-{\colim}
&
  \Set
\rlap{.}}
\]
where $p \co M \downarrow i \to \cal{A}$ is the evident projection.
The first arrow clearly preserves monomorphisms.
For the second arrow, it will suffice to verify the assumptions of \cref{reedy-mono-lemma}.
For this, we note that $p$ creates an elegant Reedy structure on $M \downarrow i$ from $\cal{A}$.
Pullbacks along face maps in $M \downarrow i$ exist and preserve face and degeneracy maps by the corresponding assumptions on $\cal{A}$, using that $i$ preserves the relevant pullbacks.
\end{proof}

\begin{proof}[Proof of \cref{left-kan-extension-pres-mono-instance}]
As per the diagram~\eqref{essential-geometric-embeddings}, the left adjoint in question is equivalent to the functor
\[
\xymatrix{
  \widehat{\Delta}
  \ar[r]^-{i_!}
&
  \widehat{\FinComplPoset}
}
\]
given by left Kan extension along the full inclusion $i \co \Delta \to \FinComplPoset$.
It will thus suffice to verify that $i$ satisfies the assumptions of \cref{left-kan-extension-pres-mono}.

We write $\Delta_\aug$ for the augmented simplex category.
Recall the elegant Reedy structures on $\Delta_\aug$ and $\Delta$, created by the full inclusion $\Delta \to \Delta_\aug$.
Note that $\Delta_\aug$ has pullbacks along face maps that preserve face and degeneracy maps and that $i' \co \Delta_\aug \to \Poset$ preserves these pullbacks.
Since $\Delta \to \Delta_\aug$ and $\FinComplPoset \to \Poset$ are fully faithful, the same holds true for $i$ whenever the limiting object of the corresponding pullback in $\Delta_\aug$ lives in $\Delta$, \ie has a point.
This is the case when the given cospan is in the image of $M \downarrow i \to \Delta$ for some $M \in \FinComplPoset$: the pullback in $\Poset$ then has a map from $M$, which has a point, for example the terminal object.
\end{proof}

\subsection{Quillen adjunctions of model structures}

Recall the Kan model structure $\widehat{\Delta}_\Kan$~\cite{quillen:kan-model-structure}, characterized by the following data:
\begin{itemize}
\item the cofibrations are the monomorphisms,
\item the fibrations are those maps with the right lifting property against horn inclusions $\Lambda^n_k \to \Delta^n$ where $n > 0$ and $0 \leq k \leq n$.
\end{itemize}
Alternatively, the fibrations can also be characterized via the right lifting property against pushout products of $\braces{k} \hookrightarrow \Delta^1$ with monomorphisms for $k = 0, 1$~\cite{gabriel-zisman:calculus-of-fractions}.%
\footnote{Note that this alternative characterization relies on classical logic to show that every monomorphism lies in the the weak saturation of boundary inclusions of simplices.
Following~\cite{sattler:model-structure}, one may reduce the use of classical logic in the development of the Kan model structure to this single point.}

In cubical sets, let us write $\Box^1$ for the functor represented by $[1]$.
Let us denote $\braces{k} \hookrightarrow \Box^1$ where $k = 0, 1$ its two points inclusions.
Using the techniques of~\cite{cohen-et-al:cubicaltt} and~\cite{sattler:model-structure}, one may show that there is a combinatorial model structure $\widehat{\Box}_\ttm$ characterized by the following data:
\begin{itemize}
\item the cofibrations are the monomorphisms,
\item the fibrations are those maps with the right lifting property against pushout products of $\braces{k} \hookrightarrow \Box^1$ with monomorphisms.
\end{itemize}
Since the details of this construction have not yet been published, we assume the existence of this model structure in the below.

In the following, we will relate this model structure to the Kan model structure via two Quillen adjunctions.
We will freely exploit the equivalence of cubical sets with presheaves over $\FinComplPoset$ and denote the essential geometric embedding $\widehat{\Delta} \to \widehat{\Box}$ simply as $i_! \dashv i^* \dashv i_*$ where $i \co \Delta \to \FinComplPoset$ is the full embedding.
This is justified by the proof of \cref{embedding-sset-to-cset}.

It is clear that $i^* \co \widehat{\Box} \to \widehat{\Delta}$ preserves cofibrations.
Since $i^*$ preserves colimits and products and sends $\Box^1$ to $\Delta^1$, it preserves prism-style generating trivial cofibrations.
The reflective embedding of simplicial sets into cubical sets thus forms a Quillen adjunction
\[
\xymatrix@C+0.5cm{
  \widehat{\Delta}_\Kan
  \ar@/^1em/[r]^{i^*}
  \ar@{}[r]|{\bot}
&
  \widehat{\Box}_\ttm
  \ar@/^1em/[l]^{i_*}
\rlap{.}}
\]
The interval objects $\Delta^1$ and $\Box^1$ induce functorial cylinders in $\widehat{\Delta}$ and $\widehat{\Box}$ via the Cartesian product (note that the latter cylinder preserves representables).
These induce evident notions of (elemetary) homotopy and homotopy equivalence in both settings.
Since $i^*$ and $i_*$ respect these interval objects and preserve products, they preserve homotopies and homotopy equivalences (it follows that $i_*$ reflects homotopy equivalences).
Note that homotopy equivalences are in particular weak equivalences and coincide with weak equivalences between fibrant objects (since all objects are cofibrant in both settings).

\begin{proposition}
We have a Quillen adjunction
\[
\xymatrix@C+0.5cm{
  \widehat{\Delta}_\ttm
  \ar@/^1em/[r]^{i_!}
  \ar@{}[r]|{\bot}
&
  \widehat{\Box}_\Kan
  \ar@/^1em/[l]^{i^*}
\rlap{.}}
\]
\end{proposition}

\begin{proof}
By \cref{left-kan-extension-pres-mono-instance}, we have that $i_!$ preserves cofibrations.
It remains to show that $i_!$ sends generating trivial cofibrations, \ie horn inclusions, to weak equivalences.

Note first that $i_! \Delta^n \simeq y[n]$ is weakly contractible as there is a contracting homotopy $[1] \times [n] \to [n]$ that sends $(0, k)$ to $0$ and $(1, k)$ to $k$.
By 2-out-of-3, it follows that the action of $i_! y$ on maps is valued in weak equivalences.

In the standard fashion of simplicial homotopy theory, one may now show by induction on $n$ that $i_!$ sends $\Lambda_{[n] \backslash I}^n \hookrightarrow \Delta^n$ to a weak equivalence where $\emptyset \subsetneq I \subsetneq \braces{0, \ldots, n}$ and $\Lambda_{[n] \setminus I}^n$ denotes the $([n] \setminus I)$-horn, the union over $i \in I$ of the $i$-th face of $\Delta^n$.
The base case $\verts{I} = 1$ of a face inclusion is covered by the preceding paragraph.
In the induction step $\verts{I} > 1$, one picks $i \in I$, has a pushout
\[
\xymatrix{
  \Lambda_{[n-1] \setminus d_i^{-1}(I)}^{n-1}
  \ar[r]
  \ar[d]
&
  \Lambda_{[n] \setminus I}^n
  \ar[d]
\\
  \Delta^{n-1}
  \ar[r]
&
  \Delta^n
  \pullbackcorner{ul}
\rlap{,}}
\]
and uses cocontinuity of $i_!$ and closure of trivial cofibrations under pushout.
The original claim is obtained for maximal choices of $I$.
\end{proof}

\bibliographystyle{alpha}
\bibliography{idempotent-completion}

\end{document}